\documentclass[11pt,notitlepage]{article}
\usepackage{amsmath,amsthm}
\usepackage{amssymb}
\usepackage{color}
\usepackage{graphicx}
\usepackage[english]{babel}

\theoremstyle{plain}
     \newtheorem{theorem}{Theorem}[section] 
     \newtheorem{conjecture}[theorem]{Conjecture}
     
     \newtheorem{lemma}[theorem]{Lemma}
     
\theoremstyle{definition}
     \newtheorem{definition}[theorem]{Definition}
     
     \newtheorem{remark}[theorem]{Remark}
 \theoremstyle{definition}

\providecommand{\R}{\mathbb{R}}
\providecommand{\norm}[1]{\lVert#1\rVert}
\author{Tiham\'er A. Kocsis and Adri\'an N\'emeth}
\title{Optimal second order diagonally implicit SSP Runge--Kutta methods}

\begin{document}
\maketitle
\begin{abstract}
Optimal Strong Stability Preserving (SSP) Runge--Kutta methods
has been widely investegated in the last decade and many open conjectures have been formulated. 
The iterated implicit midpoint rule has been observed numerically optimal in large classes of second order methods, 
and was proven to be optimal for some small cases, but no general proof was known so far to show its optimality. 
In this paper we show a new approach to analytically investigate this problem and determine the unique optimal methods in
the class of second order diagonally implicit Runge--Kutta methods.
\end{abstract}
\section{Introduction}

In this paper we investigate such numerical methods that were designed for the solution of initial value problems, such that they preserve
certain qualitative properties of the exact solutions of the differential equations with large stepsizes. Many different properties were studied in the literature such as positivity, contractivity, monotonicity, strong stability preservation and total variation diminishing (TVD) property. The first milestone was Bolley and Cruzeix's paper \cite{bolley_crouzeix}, where they proved that general linear methods cannot preserve positivity on linear problems with arbitrary large stepsize, unless they are at most first order accurate (essentially only the Backward Euler method can reach reach this infinite stepsize). 

Contractivity preservation for nonlinear systems was heavily studied by Spijker \cite{spijker_1983} proving the same order barrier for unconditional contractivity as Bolley and Crouzeix, and by van de Griend and Kraaijevanger \cite{vdg_kraaij}, Kraaijevanger \cite{kraaij} deriving computable algebraic conditions to calculate the largest feasible stepsize of the methods. Similar computable stepsize conditions were obtained by Shu and Osher for the TVD property \cite{shu_osher}, and they were investigated by Ferracina and Spijker \cite{ferr_spijker_2004, ferr_spijker_2005} and for strong stability preservation (SSP) by Gottlieb et al. \cite{gottlieb_shu_tadmor}, Gottlieb \cite{gottlieb_2005}, Higueras \cite{hig_2004, hig_2005} and for positivity preservation by Horv\'ath \cite{hz_farkas}. Stepsize conditionds for diagonally split Runge--Kutta (DSRK) methods were studied by Horv\'ath \cite{hz_1998} for positivity and numerical investigations for SSP DSRK methods were done by Macdonald et al. \cite{macd_gottlieb_ruuth}.

Extensive numerical searches were done to find optimal higher order methods with the largest stepsizes in certain classes of Runge--Kutta methods by Gottlieb and Shu \cite{gottlieb_shu},
Spiteri and Ruuth \cite{spiteri_ruuth}, Ruuth \cite{ruuth} for explicit methods. Ferracina and Spijker \cite{ferr_spijker_2008} studied SSP singly-diagonally-implicit Runge--Kutta methods and 
found the optimal methods with largest stepsizes. However, most of their optimal methods were numerically found and only a small fraction of them was proved to be optimal analytically. 
They also gave an explicit formula for the conjectured optimal coefficients for second and third order methods, and they showed that the second order method can be formulated as iterated implicit midpoint rules. Ketcheson et al. \cite{ketch_macd_gottlieb} executed a search in an even broader class of Runge--Kutta methods and the results suggested that the optimal second order methods in the class of (fully) implicit Runge--Kutta methods (IRK) is always a diagonally implicit Runge--Kutta (DIRK) method, exactly the same methods as found in \cite{ferr_spijker_2008}. Thus they extended the previous conjecture into an even stronger version: second order Runge--Kutta methods cannot have larger SSP radius than 2s, where is the number of stages in the method.

\section{Definitions, notations}

\begin{definition}\label{def:IVP} In this paper we consider initial value problems (IVPs) in a $\mathbb{V}$ vector space of form 
\begin{equation}\label{eq:IVP}
U^\prime(t)=f(U(t)),\ t\ge0, \qquad U(0)=u_0.
\end{equation}
We assume that $f:\mathbb{V}\to\mathbb{V}$ continuous and (\ref{eq:IVP}) has a unique solution $U:[0,\infty)\to \mathbb{V}$ for all $u_0\in\mathbb{V}$.
\end{definition} 

We consider Runge--Kutta methods for the numerical approximation of the solution of the IVP.

\begin{definition}\label{def:RKM} A Runge--Kutta method with $s$ stages in the Butcher form can be written as
\begin{align}\label{eq:RKM}
y_i&= U_{n}+\tau\sum_{j=1}^s a_{ij}f(y_j) \qquad i=1,2,\ldots,s \\
y_{s+1}&= U_{n} + \tau\sum_{j=1}^s b_{j}f(y_j)\\
U_{n+1}&=y_{s+1}
\end{align}
where $y_i$ are the stage values, $\tau$ is the timestep, $A=\left\{a_{ij}\right\}\in \mathbb{R}^{s\times s}$, $b=\left\{b_i\right\}\in\mathbb{R}^{s}$ are the matrices describing the method.
\end{definition} 

It can be easily shown (cf. \cite{butcher2003numerical}) that the Runge--Kutta method with matrices $A,b$ is at least second order accurate if and only if 
 \begin{align}\label{eq:OC1}
b^Te&=1\\
b^TAe&=\frac12,\label{eq:OC2}
\end{align}
where $e=(1,1,\ldots,1)^T \in \R^s$.   	 

Another popular form of Runge--Kutta methods is their Shu-Osher, writing the stage equations as linear combinations of Forward Euler steps.
 \begin{align}\label{eq:SO}
y_i&= \sum_{j=1}^s \alpha_{ij}y_j+ \beta_{ij}\tau f(y_j) \qquad i=1,2,\ldots,s+1 \\
U_{n+1}&=y_{s+1},
\end{align}
with $\alpha=\left\{\alpha_{ij}\right\}, \beta=\left\{\beta_{ij}\right\}\in\R^{(s+1)\times s}$.

Both the Butcher and Shu-Osher form have their advantages and disadvantages, however, for the sake of simplicity we here use only the $A,b$ matrices of the Butcher form, and express every
auxiliary matrix with them.
  	 
Let $\norm{\cdot}$ be an arbitary convex functional on $\mathbb{V}$, we are interested in a certain non-increasing property of $\norm{.}$, defined as below.

\begin{definition}\label{def:SSP} A Runge--Kutta method is called strong stability preserving (SSP), if the stage values and the approximation of the solution satisfy
\begin{align}\label{eq:SSP}
\norm{y_i}&\leq \norm{U_{n}} 	\qquad i=1,2,\ldots,s+1 
\end{align}
supposed that the right-hand side function $f$ in the IVP satisfy the Forward Euler condition with a timestep $\tau_0>0$, i.e.
\begin{align}\label{eq:FE}
\forall U\in\mathbb{V}\colon \ \norm{U+\tau_0F(U)}&\leq \norm{U}. 	
\end{align}
The timestep $\tau_0$ is called the Forward Euler timestep.
\end{definition}

The connections between the different concepts, such as monotonicity, contractivity, positivity, TVD, strong stability can be found in the references given in the Introduction section, or for a general review on SSP methods cf. \cite{gottlieb2011strong}. Here we just use the classical results of Kraaijevanger \cite{kraaij}, which later proved to be equivalent to many other definitions.  	 
  	 
\begin{theorem}{Kraaijevanger\cite{kraaij}, Ferracina and Spijker\cite{ferr_spijker_2008}.} If $r$ satisfy that
\begin{align}\label{eq:AMR1}
(I+rA) \textrm{ is invertible}&\\
(I+rA)^{-1}e&\geq0\\ 
rA(I+rA)^{-1}&\geq0\\
b^T(I+rA)^{-1} &\geq 0\\
1-rb^T(I+rA)^{-1}e &\geq 0,\label{eq:AMR5}
\end{align}
then the Runge--Kutta method is SSP with timestep $\tau\leq r\tau_0 $. The inequalities are meant to hold entry-wise. 
\end{theorem}

The largest $r$ satisfying (\ref{eq:AMR1}-\ref{eq:AMR5}) is called the absolute monotonicity radius, or the SSP coefficent (or SSP radius) of the Runge--Kutta method. 	 
  	 
Thus the conjecture on the optimal second order SSP methods (i.e. Conjecture 1 in \cite{ketch_macd_gottlieb}) can be formualated as:

\begin{conjecture} If an $r>0$ satisfy equations  (\ref{eq:OC1}-\ref{eq:OC2}) and  (\ref{eq:AMR1}-\ref{eq:AMR5}) for a given (fully) implicit Runge--Kutta method (IRK), then $r\leq 2s$.
\end{conjecture}  	

This conjecture was stated for singly-diagonally implicit RK (SDIRK) methods in \cite{ferr_spijker_2008} and for IRK methods in \cite{ketch_macd_gottlieb}. 
In this paper we prove this conjecture for diagonally implicit RK (DRIK) methods.

\begin{theorem} \label{thm:DIRK} If an $r>0$ satisfy equations  (\ref{eq:OC1}-\ref{eq:OC2}) and  (\ref{eq:AMR1}-\ref{eq:AMR5}) for a given diagonally implicit Runge--Kutta method (IRK), i.e. $A$ is a lower triangular matrix, then $r\leq 2s$.
\end{theorem}  	
\begin{remark}
The case when $r=2s$ can be achieved is unique, given by the following matrices

\begin{align}
  A_\mathrm{opt}={}&
  \left(
  \begin{array}{ccccc}
    \frac{1}{2s}&0&0&\cdots&0\\
    \frac{1}{s}&\frac{1}{2s}&0&\cdots&0\\
    \frac{1}{s}&\frac{1}{s}&\frac{1}{2s}&\cdots&0\\
    \vdots&&&\ddots&\\
    \frac{1}{s}&\frac{1}{s}&\cdots&\frac{1}{s}&\frac{1}{2s}
  \end{array}
  \right)\\
  b^T_\mathrm{opt}={}&
  \left(
  \begin{array}{ccccc}
      \frac{1}{s}&\frac{1}{s}&\cdots&\frac{1}{s}&\frac{1}{s}
  \end{array}
  \right).
\end{align}
The corresponding Runge--Kutta method is the iterated implicit midpoint rule, as it has been correctly conjectured previously.
\end{remark}

\begin{proof}(Theorem \ref{thm:DIRK})
Using the Butcher form, the order conditions are simple, while the absolute monotonicty inequalities contain more complicated expression. The Shu-Osher form would provide very simple absolute monotonocity inequlities, however, the order conditions would become complicated. As an intermediate solution we use our own notations to gain easily manageable inequalities.

Let us denote $N=I+rA,\ M=(I+rA)^{-1},\ w^T=rb^TN^{-1}$.
\begin{lemma}\label{lem:Nw} Theorem \ref{thm:DIRK} can be written into the following equivalent form.
For arbitrary matrices $N\in\R^{s\times}, w\in\R^{s}$, if $N$ is lower triangular and 
\begin{align}
  N &\textrm{ is invertible}\\\label{eq:OC_1}
  w^TNe &=r\\ \label{eq:OC_2}
  w^TN^2e&=\frac{r^2}{2}+r\\  
  N^{-1}e &\geq 0\\
  I-N^{-1} &\geq 0\\
  w^T &\geq 0\\
  1-w^Te &\geq 0\\  
\end{align}
hold, then $r\leq 2s$.
\end{lemma}

\begin{proof}(Lemma \ref{lem:Nw})
Simple substitutions show that $w^TNe=rb^Te$ and $w^TN^2e=rb^T(I+rA)e=rb^Te + r^2b^TAe$, thus (\ref{eq:OC1}--\ref{eq:OC2}) are equivalent to (\ref{eq:OC_1}--\ref{eq:OC_2}), provided that $N=I+rA$ is invertible. The inequalities are just simple reformulations of the absolute monotonicity inequalities, the only non-trivial one is:
\begin{align}
  I-N^{-1}=(I+rA)(I+rA)^{-1}-(I+rA)^{-1} = rA(I+rA)^{-1} 
\end{align} 
\end{proof}

From now on we work with matrices $N,M,w$. It is worth mentioning that using these new notations, the previously given $A_{opt}, b_{opt}$ matrices give the following $N_{opt}$ and $w_{opt}$.

\begin{align}
  N_\mathrm{opt}&=
  \left(
  \begin{array}{ccccc}
    2&0&0&\cdots&0\\
    2&2&0&\cdots&0\\
    2&2&2&\cdots&0\\
    \vdots&&&\ddots&\\
    2&2&\cdots&2&2
  \end{array}
  \right)\\
  w^T_\mathrm{opt}&={}
  \left(
  \begin{array}{ccccc}
      0&0&\cdots&0&1
  \end{array}
  \right)
\end{align}

We can reduce the number of the constraints by observing that
\begin{align}
\frac{(w^TNe)^2}{w^TN^2e}=\frac{r^2}{\frac{r^2}{2}+r}=\frac{1}{\frac{1}{2}+\frac{1}{r}},
\end{align}

and $r\leq 2s$ is equivalent to $\frac{1}{\frac{1}{2}+\frac{1}{r}}\leq \frac{1}{\frac{1}{2}+\frac{1}{2s}}=\frac{2s}{s+1}$,
thus it is sufficient to prove that
\begin{align}
  \frac{(w^TNe)^2}{w^TN^2e}&\leq\frac{2s}{s+1} \textrm{  subject to}\\
  w^T &\geq 0\\
  1-w^Te &\geq 0\\  
  N^{-1}e &\geq 0\\
  I-N^{-1} &\geq 0.\\
\end{align}

For any $w\geq0$, $w\neq 0$ with $w^Te<1$, one can define $w^*=\frac{w}{w^Te}$ to obtain $\frac{(w^TNe)^2}{w^TN^2e}\leq\frac{({w^*}^TNe)^2}{{w^*}^TN^2e}$ and ${w^*}^Te=1$. Thus the condition $w^Te\leq1$ can be made stronger, we can require $w^Te=1$.

Using the notation $M=N^{-1}$, we obtain that $0\leq I-N^{-1}=I-M$, therefore $M_{ij}\leq0$ for all $i\neq j$ and it is sufficient to prove that the optimum of the following problem is at most $\frac{2s}{s+1}$. (Although here we ignored the upper bound constraints for $M_{ii}$, we will see that optimum will not increase.)
\begin{align}
  \frac{(w^TM^{-1}e)^2}{w^TM^{-2}e} &\leq \frac{2s}{s+1}\ \ \textrm{  subject to}\\
  w^T &\geq 0\\
  w^Te &=1\\  
  Me &\geq 0\\
  M_{ij}&\leq 0 \ (i\neq j).
\end{align}

We can see that most of the components of the optimal $w_{opt}$ are 0. The following lemma explains this observation in general.

\begin{lemma}\label{lem:wopt}
For any fixed matrix $N\geq0$: $$\frac{(w^TNe)^2}{w^TN^2e}\leq \max_i\left\{\frac{((Ne)_i)^2}{(N^2e)_i}\right\}$$
and equality can hold only in the case when $w$ has $s-1$ zero components and one component equal to $1$.  
\end{lemma}
\begin{proof}
Apply the Cauchy-Schwarz inequality for the vectors 
$$x=\left( \frac{w_1(Ne)_1}{\sqrt{w_1(N^2e)_1}}, \frac{w_2(Ne)_2}{\sqrt{w_2(N^2e)_2}},\ldots,\frac{w_s(Ne)_s}{\sqrt{w_s(N^2e)_s}}
\right),$$
$$
y=\left(\sqrt{w_1(N^2e)_1}, \sqrt{w_2(N^2e)_2},\ldots,\sqrt{w_s(N^2e)_s}
\right).
$$
$$\frac{(w^TNe)^2}{w^TN^2e}=\frac{(x^Ty)^2}{\norm{y}^2}\leq \norm{x}^2=\sum_i w_i\frac{((Ne)_i)^2}{(N^2e)_i}\leq \max_i\left\{\frac{((Ne)_i)^2}{(N^2e)_i}\right\}$$

Equality can obviously hold only in case when $w$ has $s-1$ zero components and one component equal to $1$.
\end{proof}

We have not used so far that $M$ originates from a matrix of a diagonally implicit Runge-Kutta method, thus $M$ is lower triangular, i.e. $M_{ij}=0$ for $i<j$.
$Me\geq 0$ means that the row sums of $M$ are nonnegative, $M_{ij}\leq0$ for $i\neq j$,  thus $M_{ii}\geq0$ must hold. Since $M$ is invertible $M_{ii}>0$ holds, too.
The constraints and objective are positively homogeneous, i.e. invariant under the transformation $M:= cM$, for any $c>0$, so we can assume that $M_{ii}=1$.

First we analyze the $s=2$ case. Although it was proven by many times in the literature, we provide a simple proof here to illustrate the advantages of our formalization.

\begin{lemma}\label{lem:dirk_s2}
The optimal second order two-stage DIRK method has $r=4$ with the coefficients

$$M=
\left(
\begin{array}{cc}
  1&0\\
  -1&1
\end{array}\right),\ w^T=(0,1).$$ 

\end{lemma}
\begin{proof}
$M=
\left(
\begin{array}{cc}
  a&0\\
  -c&1
\end{array}
\right)$ with $a>0,c\geq 0$, $-c+1\geq0$ and $0\leq w^T=(w_1,w_2)$, $w_1+w_2=1$. 
$$N=M^{-1}=
\left(
\begin{array}{cc}
  \frac1{a}&0\\
  \frac{c}{a}&1
\end{array}
\right), \ \ N^2=M^{-2}=
\left(
\begin{array}{cc}
  \frac1{a^2}&0\\
  \frac{c}{a}+\frac{c}{a^2}&1
\end{array}
\right)$$

$$\frac{(w^TNe)^2}{w^TN^2e}=\frac{\left(\frac{c}{a}+1\right)^2}{\left(\frac{c}{a}+\frac{c}{a^2}+1\right)}
\leq\frac43,$$
because
$$3\left(1+\frac{c}{a}\right)^2-4\left(1+\frac{c}{a^2}+\frac{c^2}{a^2}\right)=-\left(1-\frac{c}{a}\right)^2+\frac{4c^2-4c}{a^2}\leq0.$$
Equality can only hold in the case when $c=1$, $a=1$ and $w^T=(0,1)$. 
\end{proof}

Now we prove the general case using mathematical induction. We have proved our theorem for $s=2$ and it is trivial for $s=1$.
Assume that $\frac{(w^TM^{-1}e)^2}{w^TM^{-2}e}\leq\frac{2k}{k+1}$ holds for any $w\in\mathbb{R}^{k}$, $M\in\mathbb{R}^{k\times k}$ with $0\leq w$, $w^Te=1$, $M_{ij}\leq0$ $\ (i\neq j)$, $Me\geq 0$ for all $k=1,2,\ldots, s-1$.

As a consequence of the Lemma, we can now assume that there exists an index $j$: $w_j=1$, $w_i=0$, $(i\neq j)$.
Thus $w^TNe$ is the sum of the $j$-th row of $N=M^{-1}$, $w^TN^2e$ is the sum of the $j$-th row of $N^2$.
Since $M$ is a lower triangular matrix, the $j$-th rows of $M^{-1}$ and $M^{-2}$ only depend on the upper left $j \times j$ minor of $M$. Therefore if $j<s$, we can apply the assumption of the mathematical induction: 

$$\frac{(w^TNe)^2}{w^TN^2e}\leq \frac{2j}{j+1}<\frac{2s}{s+1}.$$

The only remaining case is $j=s$, thus we need to compute the sum of the last rows of $N$ and $N^2$. 

Now let us consider a matrix $M_s\in\mathbb{R}^{s\times s}$ in the partitioned form: $M_s=\left(
\begin{array}{cc}
  M_{s-1}&0\\
  -a^T&1
\end{array}\right)$ and a vector $w\in\mathbb{R}^{s}$ in the form $w_s^T=(w_{s-1}^T, w)$. We suppose that $M_s, w$ do not violate the constraints of the optimization problem.

The inverse of the partitioned matrix can be explicitly calculated as 
$$M_{s}^{-1}= \frac{1}{\det(M_s)}\left(
\begin{array}{cc}
  M_{s-1}^{-1}&0\\
  a^TM_{s-1}^{-1}&1
\end{array}\right), \ 
M_s^{-2}=\frac{1}{\det(M_s)^2}
\left(
\begin{array}{cc}
  M_{s-1}^{-2}&0\\
  a^T(M_{s-1}^{-1}+M_{s-1}^{-2})&1
\end{array} \right).
$$
Thus the expression we have to maximize $(w^T=(0,0,\ldots,1))$

$$\frac{(w^TNe)^2}{w^TN^2e}=\frac{(1+a^TM_{s-1}^{-1}e)^2}{1+a^TM_{s-1}^{-1}e+a^TM_{s-1}^{-2}e}$$
Applying the assumption of the induction for $a\in\mathbb{R}^{s-1}$, $M_{s-1}\in\mathbb{R}^{(s-1)\times (s-1)}$:

$$\frac{(a^TM_{s-1}^{-1}e)^2}{a^TM_{s-1}^{-2}e}\leq\frac{2(s-1)}{s}$$

Denote $0\leq\alpha=a^TM_{s-1}^{-1}e$, $0\leq\beta=a^TM_{s-1}^{-2}e$. It is sufficient to prove the following lemma.

\begin{lemma}\label{lem:alphabeta}
\begin{align*}
 \frac{(1+\alpha)^2}{1+\alpha+\beta}&\leq\frac{2s}{s+1},  \textrm{ whenever} \\
\frac{\alpha^2}{\beta}&\leq \frac{2(s-1)}{s}
\end{align*}
\end{lemma}

\begin{proof}
$$\frac{(1+\alpha)^2}{1+\alpha+\beta}\leq \frac{(1+\alpha)^2}{1+\alpha+\frac{s}{2(s-1)}\alpha^2}$$
This fraction has the value $1$ at $\alpha=0$, its limit is $\frac{2(s-1)}{s}<\frac{2s}{s+1}$ at $\alpha\to\infty$.
It has only one local minimum on the interval $[0,\infty)$, namely $\alpha=s-1$, because its derivative is:

$$\frac{4(s-1-\alpha)(\alpha+1)(s-1)}{(2s-2a-2+2s\alpha+s\alpha^2)^2}. $$

The value of the function at $\alpha=s-1$ is $\frac{2s}{s+1}$, and that is the global maximum of the function, which proves the lemma.
\end{proof}

One can also calculate the optimal $M_s$ matrix by analyzing the sharpness of our estimations and using the assumptions of the induction, obtaining a bidiagonal matrix: $M_{i,i-1}=-1$, $M_{ii}=1$ are the only nonzero elements of $M$ (this easily follows from the fact that $a^T=(0,0,\ldots,1)$ must hold to have equality in the last estimation).

Thus the step in the mathematical induction is working, we proved that $r\leq 2s$ holds for diagonally implicit Runge--Kutta methods and the optimal method is unique, proving Theorem \ref{thm:DIRK}.

\end{proof}

\section*{Acknowledgment}

This research was supported by the project
T\'AMOP-4.2.2.A-11/1/KONV-2012-0012: Basic research for the development of hybrid and electric vehicles - The Project is supported by the Hungarian Government and co-financed by the European Social Fund.

\bibliography{2s}{}
\bibliographystyle{siam}

\end{document}